\documentclass{amsart}
\usepackage[T1]{fontenc}
\usepackage[utf8]{inputenc}
\usepackage{latexsym}
\usepackage{amsfonts}
\usepackage{amsmath}
\usepackage{amssymb}
\usepackage{amsthm}
\usepackage{extpfeil}
\usepackage{verbatim}
\usepackage[nobysame]{amsrefs}
\usepackage{xcolor}
\definecolor{darkred}{RGB}{180,0,0}
\definecolor{darkblue}{RGB}{0,0,180}
\usepackage{hyperref}
\hypersetup{
  colorlinks = true,
  urlcolor   = darkblue,
  linkcolor  = darkblue,
  citecolor  = darkred,
  pdftitle={Warped cones over profinite completions}
}
\usepackage{mathtools}
\usepackage{ulem}
\usepackage{enumitem}
\usepackage{bm}
\usepackage{extarrows}
\usepackage{upref}
 \usepackage{amsaddr}

\newtheorem{thm}{Theorem}[section]
\newtheorem{fact}[thm]{Fact}
\newtheorem*{fact*}{Fact}
\newtheorem*{thm*}{Theorem}
\newtheorem{lem}[thm]{Lemma}
\newtheorem{rem}[thm]{Remark}
\newtheorem{prop}[thm]{Proposition}
\newtheorem{cor}[thm]{Corollary}
\newtheorem{ex}[thm]{Example}
\newtheorem{question}[thm]{Question}
\theoremstyle{definition}
\newtheorem{defi}[thm]{Definition}
\newtheorem{conv}[thm]{Convention}

\newtheorem{longEx}[thm]{Example}
\newtheorem{longrem}[thm]{Remark}

\def\R{{\mathbb R}}
\def\N{{\mathbb N}}
\def\Z{{\mathbb Z}}

\def\F{{\mathbb F}}
\def\T{{\mathbb T}}

\def\H{{\mathcal H}}

\def\cO{{\mathcal O}}

\def\eps{\varepsilon}

\def\onto{\twoheadrightarrow}
\def\st{\ |\ }
\def\Bigst{\ \Big|\ }

\DeclareMathOperator{\supp}{supp}

\DeclareMathOperator{\diam}{diam}

\newcommand{\map}[3]{{#1 \colon #2 \to #3}}
\newcommand{\ceil}[1]{{\left\lceil #1 \right\rceil}}
\newcommand{\floor}[1]{{\left\lfloor #1 \right\rfloor}}
\newcommand*{\defeq}{\mathrel{\vcenter{\baselineskip0.5ex \lineskiplimit0pt
                     \hbox{\scriptsize.}\hbox{\scriptsize.}}}%
                     =}
\newcommand*{\eqdef}{=\mathrel{\vcenter{\baselineskip0.5ex \lineskiplimit0pt
                     \hbox{\scriptsize.}\hbox{\scriptsize.}}}}

\title{Warped cones over profinite completions}
\author{Damian Sawicki}
\address{Institute of Mathematics of the Polish Academy of Sciences \\
 Śniadeckich 8, 02-097 Warszawa, Poland}

\begin{document}

\maketitle

\begin{abstract}
We construct metric spaces that do not have property A yet are coarsely embeddable into the Hilbert space. Our examples are so called warped cones, which were introduced by J. Roe to serve as examples of spaces non-embeddable into a Hilbert space and with or without property A. The construction provides the first examples of warped cones combining coarse embeddability and lack of property A.

We also construct warped cones over manifolds with isometrically embedded expanders and generalise Roe's criteria for the lack of property A or coarse embeddability of a warped cone. Along the way, it is proven that property A of the warped cone over a profinite completion is equivalent to amenability of the group.

In the appendix we solve a problem of Nowak regarding his examples of spaces with similar properties.
\end{abstract}

%---------------------------------------------------------

\section*{Introduction}
\subsection*{Warped cones}
Let $Y$ be a compact subset of the Euclidean sphere $S^{n-1}\subseteq \R^n$ admitting an action of a finitely generated group $\Gamma$. The warped cone of $Y$ consists of the set $\cO Y = \{ty \st t\in(0,\infty),\ y\in Y \} \subseteq \R^n$ with the Euclidean metric properly ``warped'' by the group action.

Warped cones were introduced by John Roe in \cites{foliatedCones,cones}, where they serve as a rich source of examples of spaces with and without property~A of Guoliang Yu \cite{Yu}. He also describes a class of warped cones that do not coarsely embed into a Hilbert space. Warped cones were also studied in \cites{Kim, Wulff}.

While giving sufficient conditions for the property A, its negation and the negation of coarse embeddability, Roe leaves the problem of finding sufficient conditions for the coarse embeddability open, which was a motivation for this work.

The group action on a cone becomes an action by translations on the warped cone, so it yields finite propagation operators in the uniform Roe algebra. Recently \cite{DN}, Dru\c{t}u and Nowak constructed from them so called non-compact ghost projections. Such projections are related to the coarse Baum--Connes conjecture and it is suspected \cites{cones,DN} that some warped cones are counterexamples to the conjecture.

\subsection*{Property A and coarse embeddability}

Property~A was designed by Yu to imply coarse embeddability into a Hilbert space and whether the converse is true has been a question absorbing the researchers for the last decade. For finitely generated groups it was shown in \cite{compression} that existence of a sufficiently good embedding (with \emph{compression} larger than a half) implies property~A. 

The first examples of spaces without property~A but embeddable into a Hilbert space were provided by Nowak in \cite{Nowak}. The first family of examples with bounded geometry was found by Arzhantseva, Guentner, and \v{S}pakula \cite{AGS}, and their construction was vastly generalised by Khukhro \cite{Khukhro}. Additional examples are provided thanks to the permanence results of \cites{Khukhro12, CDK} due to Cave, Dreesen, and Khukhro. All these examples are coarse disjoint sums of finite Cayley graphs. Recently, the problem was finally solved for infinite groups rather than sequences of finite metric spaces, when Osajda \cite{Osajda}, building on the preceding work \cite{AO} with Arzhantseva, constructed examples being a finitely generated group.

\subsection*{Results}

We construct the first warped cones for which coarse embeddability is obtained without property A (Theorem~\ref{T:summary}), extending the abovementioned type of results to the realm of warped cones. Moreover, we show that for the considered class of warped cones property A is equivalent to amenability of the group (\ref{T:A=Amen}). Our construction -- given an appropriate sequence of subgroups of a non-amenable group -- provides an example of a warped cone that is embeddable into a Hilbert space yet does not satisfy property~A. Families of
such sequences are the main results of \cites{AGS, Khukhro12, Khukhro}.

We also generalise (\ref{P:AtoAmenability}) two results of \cite{cones}, which assert that coarse embeddability (respectively: property A) of the cone implies the Haagerup property (respectively: amenability) of the group $\Gamma$.

Apart from that, we describe sufficient conditions (\ref{P:embeddingExpanders}) for the existence of isometrically embedded expanders in the warped cone, illustrating it by examples coming from the action of linear groups on tori.

Along the way to the main theorems, we develop the theory of warped cones. We provide (easier to check) conditions (\ref{L:coneEmbeddability}, \ref{P:propA}) equivalent to coarse embeddability and property A for warped cones. Furthermore, we propose a formula (\ref{P:quasimetric}) that significantly simplifies calculation of the warped metric.

Finally, in the appendix, we solve a problem posed by Nowak \cite{Nowak} about his examples of coarsely embeddable spaces without property A. This enlarges the family of spaces to which his theorem can be applied.

%---------------------------------------------------------

\section{Warped metric and cones}\label{S:WarpedMetricAndCones}
The following definition and fact from \cite{cones} are essential.

\begin{defi} Let $(X,d)$ be a metric space and $\Gamma$ be a group acting by homeomorphisms on $X$, provided with a finite generating set $S$. The warped metric $d_\Gamma$ on $X$ is the greatest metric satisfying the following inequalities:
\begin{eqnarray*} d_\Gamma(x,x')\leq d(x,x'), & & d_\Gamma(x,sx)\leq 1 \ \forall s\in S.
\end{eqnarray*}
\end{defi}

\begin{fact}\label{F:metric} Let $x,x'\in X$. For $\gamma\in\Gamma$ let $|\gamma|$ denote the word length of $\gamma$ relative to the generating set $S$. The warped distance $d_\Gamma(x,x')$ is the infimum of sums
$$\sum |\gamma_i| +  d(\gamma_i x_i, x_{i+1})$$
taken over all finite sequences $x=x_0, x_1, \ldots, x_N=x'$ in $X$ and $\gamma_0,\ldots, \gamma_{N-1}$ in $\Gamma$. Moreover, if $d_\Gamma(x,x')\leq k$, then it is enough to consider sequences of length at most $N=k$ and if $X$ is proper, then there is a sequence such that the infimum is attained.
\end{fact}

One can also observe that it is enough to consider $\gamma_i$ of length at most~$1$.

Note that the construction relies on the generating set of $\Gamma$ and the metric on $X$, but the coarse structure of the warped space remains the same if we choose another finite generating set or a different but coarsely equivalent metric on $X$.

We can now introduce the cones.

\begin{defi}\label{D:warpedCone} Let $(Y,d_Y)$ be a compact space of diameter at most $2$, admitting a continuous action of a finitely generated group $\Gamma$. The cone of $Y$ is the metric space $\cO Y = ([1,\infty)\times Y,\ d)$, where metric $d$ is defined as follows:
$$d((t,y),(t',y')) = |t-t'|+\min(t,t')\cdot d_Y(y,y').$$
The warped cone of $Y$, denoted by $\cO_\Gamma Y$, is the cone of $Y$ with the warped metric, where the warping group action is defined by the formula: $\gamma(t,y) = (t,\gamma y)$.
\end{defi}

We ignore the tip of the cone for the sake of notational simplicity, but, clearly, considerations for the whole cone $[0,\infty)\times Y / \{0\}\times Y$ rather than $[1,\infty)\times Y$ would be the same.

One easily verifies that indeed the above formula gives a metric (we prove it in a more general case in Lemma \ref{L:metricWarpedCone}) and if $Y$ is a subset of the unit sphere in a Banach space $V$ (primarily, $Y
\subseteq S^n\subset \R^{n+1}$), then the map $\cO Y\owns (t,y)\mapsto ty\in V$ establishes a bi-Lipschitz homeomorphism onto the image. This image is a geometric model of our space.

We would like to mention the following results from \cite{cones}.

\begin{fact} If $(X,d)$ is proper, then the warped space $(X,d_\Gamma)$ is also proper. In particular, it happens for warped cones.
\end{fact}

\begin{fact} If $(X,d)$ has bounded geometry and the action of $\Gamma$ on $X$ is Lipschitz, then the warped space $(X,d_\Gamma)$ also has bounded geometry. In particular, if $\Gamma$ acts Lipschitzly on $Y$ and $Y$ embeds bi-Lipschitz in $\R^n$, then the warped cone $\cO_\Gamma Y$ has bounded geometry.
\end{fact}

Here, bounded geometry is considered in the sense of \cite{indexTheory} -- a metric space has bounded geometry if it is coarsely equivalent to a discrete space with uniform bounds on cardinalities of balls of any fixed radius.

Below, we present a formula for the warped metric on the cone, which is very similar to the formula for the non-warped cone metric.

\begin{lem}\label{L:metricWarpedCone} For $s\geq 1$ let $d_s$ denote the warped metric on $Y$ originating from the metric $sd_Y$. The warped metric on the cone of $Y$ is given by the following formula, where $t\geq 0$:
$$d_\Gamma((s,y),(s+t,y')) = t + d_s(y,y').$$
\end{lem}
\begin{proof}
Clearly, we have the inequality:
\begin{equation*}
d_\Gamma((s,y), (s+t,y')) \leq t + d_s(y,y')
\end{equation*}
and would like to show that it is an equality.

Let $(s_i,y_i)_{i=0}^N$, $(\gamma_i)_{i=0}^{N-1}$ be sequences minimising the formula for the warped distance between $(s,y)$ and $(s+t,y')$, and denote $r=\min_i(s_i)$. This distance is equal to:
\begin{multline}\label{equation2}
\sum_i \big( |\gamma_i| + |s_{i+1} - s_i| +  \min(s_i,s_{i+1})\cdot d_Y(\gamma_i y_i, y_{i+1}) \big) \\
\geq \sum_i |\gamma_i| + t + 2|s-r| + r\cdot\sum_i d_Y(\gamma_i y_i, y_{i+1}).
\end{multline}

Clearly $\sum_i d_Y(\gamma_i y_i, y_{i+1}) \leq 2$ (otherwise, the sequence $((s,y), (r,y), (r,y'), (s+t,y'))$ with $\gamma \equiv 1$ would be shorter than the given one), so the right hand side of \eqref{equation2} is at least:
$$\geq \sum_i |\gamma_i| + t + (|s-r| + r)\cdot \sum_i d_Y(\gamma_i y_i, y_{i+1}) = t + \sum_i \big(|\gamma_i| + sd_Y(\gamma_i y_i, y_{i+1})\big),$$
which in turn is no smaller than $t+d_s(y,y')$ by Fact \ref{F:metric}.
\end{proof}

%---------------------------------------------------------

\section{A non-metric coarsely equivalent to the warped metric}\label{S:quasimetric}

Below, we show that if the $\Gamma$-action is isometric, then the formula for the warped metric can be vastly simplified. If the action is Lipschitz, then the obtained function is (in general) not equal to the warped metric, actually it may even be not symmetric, but it is coarsely equivalent to the warped metric.

\begin{prop}\label{P:quasimetric} Suppose that $\Gamma$ is a group acting on $(X,d)$ and $L\geq1$ is the bi-Lipschitz constant for the action of the generators $s\in S$ of $\Gamma$. Let $D_\Gamma(x,x') = \inf_{\gamma\in\Gamma}( |\gamma| + d(\gamma x, x'))$. Then the following inequality holds for all $x,x'\in X$:
$$d_\Gamma(x,x')\leq D_\Gamma(x,x') \leq L^{{d_\Gamma(x,x')}}\cdot d_\Gamma(x,x').$$
In particular $D_\Gamma$ is equal to $d_\Gamma$ if the action is isometric and ``coarsely equivalent'' in the general case.
\end{prop}

\begin{proof} The warped metric is the infimum over pairs of finite sequences $(x_i)_{i=0}^{M+1}\subseteq X$, $(\gamma_i)_{i=0}^{M}\subseteq \Gamma$ (where $x_0=x$ and $x_{M+1}=x'$) of the sums $\sum |\gamma_i| + d(\gamma_i x_i, x_{i+1})$. Function $D_\Gamma$ is the infimum over sequences of length $M=0$, so the inequality $d_\Gamma \leq D_\Gamma$ follows.

For the converse, for a fixed pair of sequences $(x_i)_{i=0}^{M+1}$, $(\gamma_i)_{i=0}^{M}$ let us denote $B_k = \sum_{i=0}^k|\gamma_i|$ and $\gamma^k = \gamma_k\cdot\ldots\cdot\gamma_0$, where $0\leq k \leq M$. By induction over $M$ we will show:
$$ B_{M} + \sum_{i=0}^{M} d(\gamma_i x_i, x_{i+1})
\geq
|\gamma^{M}| + L^{-B_{M}} d(\gamma^{M}x_0, x_{M+1})$$
and hence -- as we can assume $B_{M}\leq d_\Gamma(x_0, x_{M+1})$ -- it follows that
$$d_\Gamma(x,x')\geq L^{-{d_\Gamma(x,x')}} D_\Gamma(x,x').$$

For $M=0$ the desired inequality is straightforward. The inductive step works as follows:
\begin{align*}
B_{M} + \sum_{i=0}^{M} d(\gamma_i x_i, x_{i+1})
= B_{M-1} + \sum_{i=0}^{M-1} d(\gamma_i x_i, x_{i+1}) + |\gamma_M| + d(\gamma_M x_M, x_{M+1}) \\
\geq |\gamma^{M-1}| + L^{-B_{M-1}} d(\gamma^{M-1}x_0, x_{M}) + |\gamma_M| + d(\gamma_M x_M, x_{M+1}) \\
\geq  |\gamma^{M-1}| + |\gamma_M| + L^{-B_{M-1}} \Big(d(\gamma^{M-1}x_0, x_{M}) + d(\gamma_M x_M, x_{M+1})\Big) \\
\geq |\gamma^{M}| + L^{-B_{M-1}} \Big(L^{-|\gamma_M|}d(\gamma_M\gamma^{M-1}x_0, \gamma_M x_{M}) + d(\gamma_M x_M, x_{M+1})\Big) \\
\geq |\gamma^{M}| +  L^{-B_{M}} d(\gamma^M x_0, x_{M+1}).\ \ \ \qedhere
\end{align*}
\end{proof}

\begin{longrem}\label{R:finitedefinition} Due to the above bound and the obvious bound $D_\Gamma(x,x') \leq d(x,x')$, in the definition of $D_\Gamma$ it is enough to take the minimum over
$$|\gamma|\leq {\min\left(L^{{d_\Gamma(x,x')}}\cdot d_\Gamma(x,x'),\ d(x,x')\right)}.$$
\end{longrem}

\begin{cor} Assume that $\Gamma$ acts Lipschitzly on $(X,d)$ and $Z$ is a coarsely dense $\Gamma$-subset of $X$. Denote the restricted metric by $d'$. Then, the coarse spaces $(Z,{d'}_\Gamma)$ and $(X,d_\Gamma)$ are equivalent via the inclusion.
\end{cor}
\begin{proof} By Proposition \ref{P:quasimetric} we have not only $(Z,{d'}_\Gamma) \simeq (Z,D_\Gamma)$, but also $(Z,{d_\Gamma}_{|Z^2}) \simeq (Z,D_\Gamma)$. Combining it with the fact that $(Z,{d_\Gamma}_{|Z^2})$ is a coarsely dense subspace of $(X,d_\Gamma)$ ends the proof.
\end{proof}

In particular, if $Y$ contains a dense orbit $Y'$, then $\cO_\Gamma Y' \simeq \cO_\Gamma Y$ (actually, this is an isometric embedding even if the action is not Lipschitz). One may also note that in the above corollary it is enough to assume that $Z$ is coarsely dense in the \emph{warped} space $X$.

The formula for $D_\Gamma$ considered in the language of the formula for $d_\Gamma$ (Fact \ref{F:metric}) can be viewed as the infimum over sequences of length $N=M+1=1$. Provided that $X$ consists of one orbit, one can go further and consider a function being the infimum over ``sequences of length one half'': $\Delta_\Gamma(x,x') = \min \big(d(x,x'), \inf_{\gamma x=x'}|\gamma|\big)$. The following example shows that $\Delta_\Gamma$ is no longer equivalent to $d_\Gamma$, even if $X$ consists of one orbit.

\begin{ex} Let $\Gamma=\Z^2$ act on the set $X=\{(2^n,2^m) \st n,m\in \Z\}$ in a Lipschitz way by $(k,l)\cdot(2^n,2^m) = (2^{n+k},2^{m+l})$. Function $\Delta_\Gamma$ is not coarsely equivalent to the warped metric $d_\Gamma$.
\end{ex}
\begin{proof}Let $x_n=(2^n, 2^{-n})$, $x_n'=(2^{n+1}, 1)$ for $n\geq 1$. Then $d_\Gamma(x_n,x_n')=1 + (1-2^{-n+1})\leq 2$, but $\Delta_\Gamma(x_n,x_n') = \min(2^n+1-2^{-n}, 1+n) \xlongrightarrow{n\to \infty} \infty$.
\end{proof}

%---------------------------------------------------------

\section{Warped cones with isometrically embedded expanders}\label{S:expanders}

\subsection*{\texorpdfstring{How the distance of $\boldsymbol{y}$ and $\boldsymbol{y'}$ changes when we move deep into the cone}{How the distance of y and y' changes when we move deep into the cone}}

Let $(Y,d)$ be a proper (e.g., compact as before) metric space with a continuous action of a group $\Gamma$ with a fixed finite symmetric set of generators $S$ containing the identity element. For $y,y'\in Y$, define $\delta_n(y,y')$ as
$$\min \left\{\sum_{i=0}^n d(y_{2i}, y_{2i+1}) \Bigst y_0=y,\ y_{2n+1}=y',\ \forall_{1\leq i\leq n} \exists_{s\in S}: y_{2i}=sy_{2i-1} \right\}.$$
This is the ``$Y$-part'' of the distance formula from Fact \ref{F:metric}, if we limit the ``$\Gamma$-part'' (length of sequences) by $n$. Hence, $d_\Gamma(y,y') = \min_{n\in \N} (n+\delta_n(y,y'))$. Finally, if by $d_s$ we denote the warped metric on $(Y,sd)$, we get
$$d_s(y,y')=\min_{n\in \N} (n+s\cdot \delta_n(y,y')).$$

The last formula shows in particular that $s\mapsto d_s(y,y')$ is a concave function (as the minimum of concave functions). More interestingly, we obtain the following.

\begin{longrem}\label{r:metricConvergence} If $y,y'$ lie in different orbits of the $\Gamma$-action -- which is equivalent to $\delta_n(y,y')$ being positive for all $n$ -- the distance $d_s(y,y')$ goes to infinity with $s$. Contrary, if they belong to the same orbit and $N$ is the length of the shortest $\gamma$ such that $y' = \gamma y$, we have $\delta_n(y,y') = 0$ for $n\geq N$ and $d_s(y,y') = N$ for $s\geq \max_{n<N}\frac{N-n}{\delta_n(y,y')}$.
\end{longrem}

The above remark is a quantitative version of the observation present in \cite{cones}*{Lemma 1.9}.
Effective calculations (or estimates) of $\delta_n$ could improve our understanding of warped cones.

\subsection*{Warped metric on an orbit} Any orbit $\Gamma y$ is isomorphic as a $\Gamma$-set with the set of left cosets $\gamma H$, where $H$ is the stabiliser of $y$. This set of cosets is by the definition
a quotient of $\Gamma$ and hence it has a quotient metric of the right-invariant metric on $\Gamma$: the distance between $\gamma H$ and $\gamma'H$ is the length of the shortest $\gamma''$ such that $\gamma''\gamma' H = \gamma H$. Thus, by Remark \ref{r:metricConvergence}, for any finite subset $F$ of any orbit $\Gamma y$ and $s$ sufficiently large, the warped metric $d_s$ restricted to $F$ is equal to the respective restriction of the quotient metric on $\Gamma /H$.

In particular, if there is a sequence of finite index subgroups $\Gamma_m<\Gamma$ such that $Y$ contains orbits isomorphic to $\Gamma/\Gamma_m$, then the coarse disjoint sum  $\bigsqcup_m\Gamma/\Gamma_m$ embeds isometrically in $\cO_\Gamma Y$.

Recall that if $\Gamma$ has Kazhdan's property (T), and $\Gamma_m$ is a sequence of its subgroups with increasing finite index, then $G_m=\Gamma/\Gamma_m$ is a sequence of expanders. Hence, we get the following.

\begin{prop}\label{P:embeddingExpanders} Let $\Gamma$ be a countable property (T) group acting on $Y$ in such a way that there are arbitrarily large finite orbits. Then, $\cO_\Gamma Y$ contains an isometrically embedded sequence of expanders.
\end{prop}

\subsection*{Concrete example} A standard example of groups with Kazhdan's property (T) are special linear groups $\mathrm{SL}_n(\Z)$ in dimension $n\geq 3$. Their obvious actions on $\R^n$ pass to quotient spaces $\T^n = \R^n/\Z^n$.
Clearly, the orbit of any point $x=\big(\frac{p_i}{q_i}\big)_{i=1}^n$ with rational coordinates is finite as contained in the set of all points of the form $\big(\frac{k_i}{LCM(q_1,\ldots,q_n)}\big)_{i=1}^n$, which has $LCM(q_1,\ldots,q_n)^n$ elements.

Moreover, the point $x_k=\left(\frac{p_1}{q_1}, \frac{p_2}{q_2} + \frac{kp_1}{q_1}, \frac{p_3}{q_3},\ldots, \frac{p_n}{q_n} \right)$ belongs to the orbit of $x$ for any $k\in \Z$, and -- if $p_1$ and $q_1$ are coprime, then $x_k$ are distinct for all $1\leq k\leq q_1$. Hence, the cardinality of the orbit of $x$ increases to infinity with $q_1$. By Proposition \ref{P:embeddingExpanders} we get the following.

\begin{ex}
$\cO_{\mathrm{SL}_n(\Z)}\T^n$ contains an isometrically embedded sequence of expanders.
\end{ex}

One often prefers a nested chain of subgroups $\Gamma_m$ with trivial intersection $\bigcap \Gamma_m=\{1\}$. It can be obtained with a careful choice of points $x_m$ as follows.

\begin{longEx}\label{E:expanders2}
Let $(q_i)_{i=1}^n$ be pairwise coprime numbers (greater than $1$) and consider $x_m=(q_i^{-m})_{i=1}^n$ for $m\geq 1$. Let $A=(a_{i,j})_{i,j=1}^n$ be any matrix with integer coefficients. Then, $x_m$ is a fixed point of $A$ if and only if the following equivalent congruences hold for all $i$:
\begin{eqnarray*}
q_i^{-m} & \equiv & \sum_j a_{i,j}q_j^{-m} \pmod 1\\
\prod_{l\neq i} q_l^m & \equiv & \sum_j a_{i,j}\prod_{k\neq j}q_k^m \pmod{ \prod_l q_l^m}
\end{eqnarray*}
By the Chinese remainder theorem the last congruence is equivalent to the following conjunction of congruences:
$$\forall{t\in\{1,\ldots,n\}}\ \ \prod_{l\neq i} q_l^m \equiv \sum_j a_{i,j}\prod_{k\neq j}q_k^m \pmod{  q_t^m},$$
and the right hand side of the above is equal to
$$\sum_{j\neq t} a_{i,j}\prod_{k\neq j}q_k^m + a_{i,t}\prod_{k\neq t}q_k^m \equiv a_{i,t}\prod_{k\neq t}q_k^m,$$
so we obtain
\begin{equation*}
\left\{\ 
\begin{aligned}
0 &\equiv& a_{i,t} \prod_{k\neq t}q_k^m \pmod{q_t^m} & \text{\ \ for } t\neq i\\
\prod_{l\neq i} q_l^m &\equiv& a_{i,i} \prod_{k\neq i}q_k^m \pmod{q_i^m}. &
\end{aligned}
\right.
\end{equation*}
Since $\prod_{k\neq t} q_k^m$ is invertible in $\Z_{q_t^m}$, it follows that the above is equivalent to $a_{i,t} \equiv 0 \pmod{q_t^m}$ for $t\neq i$ and $a_{i,i} \equiv 1 \pmod{q_i^m}$. It follows that the sequence of stabilisers $\Gamma_m\defeq \operatorname{Stab}(x_m)$ is nested and intersects trivially.
\end{longEx}

%---------------------------------------------------------

\section{Coarse embeddability}\label{S:embeddability}

In this and the consecutive section we will prove auxiliary results asserting that coarse embeddability and property A of the warped cone can be obtained from coarse embeddability and property A of its slices.

Let us recall (see, for example, \cite{Roe}*{Chapter 11}) that a coarse embedding is such a function $f\colon X\to Y$ between metric spaces that there are increasing to infinity functions $\rho_-,\rho_+\colon \R_+\to \R$ such that:
$$\rho_-\circ d_Y(f(x),f(x'))\leq d_X(x,x') \leq \rho_+\circ d_Y(f(x),f(x')).$$
A coarse embedding $f$ is a quasi-isometric embedding, if one can choose $\rho_-$, $\rho_+$ to be affine.

\begin{lem}\label{L:coneEmbeddability} A warped cone $\cO_\Gamma Y = [1,\infty)\times Y$ embeds coarsely into a Hilbert space if and only if all of its sections $\{s\}\times Y$ do with the same functions $\rho_-$ and $\rho_+$.
\end{lem}

\begin{proof} The ``only if'' part is obvious.

For the ``if'' part, it is enough to consider sections of the form $\{2^n\}\times Y$. Let $\phi_n\colon \{2^n\}\times Y \to \H_n$ be the coarse embedding into a Hilbert space. Since we have $\diam \{2^n\}\times Y \leq 2^n \cdot \diam Y$ and any coarse embedding into a Hilbert space can be modified to obtain a coarse embedding with arbitrarily small $\rho_+$ \cite{growthOfCocycles}*{Proposition 3.10}, we can assume that every $\phi_n$ maps into a ball in $\H_n$ of radius at most $D\cdot2^n$ for some $D<\infty$.

The target space for our embedding $\Phi$ will be $\R \oplus \bigoplus_n \H_n$. We write any $s\in [2^n,2^{n+1}]$ as a convex combination $s = \theta_s \cdot 2^n + (1-\theta_s) \cdot 2^{n+1}$, and put 
$$\Phi(s,y)= s + \theta_s\cdot \phi_n(y) + (1-\theta_s)\cdot \phi_{n+1}(y).$$

Let $(s,y), (t,y')\in \cO_\Gamma Y$. Let us first check the case $t/s\geq 2$, that is, for some $n$, we have $s\leq 2^n \leq 2^{n+1} \leq t\leq 2^{n+2}$. Recall (Definition \ref{D:warpedCone}) that we assume $\diam Y\leq 2$.
\begin{multline*}
\frac{d_\Gamma((s,y),(t,y'))}{8}\leq
\frac{1}{2}\left(\frac{t-s}{4} + \frac{d_s(y,y')}{2}\right)\leq \frac{1}{2}(2^n+2^{n-1}\cdot\diam Y) \leq 2^n\\\leq t-s
\leq \|\Phi(t,y') - \Phi(s,y)\| 
\leq t-s + D\cdot 2^{n} + D\cdot 2^{n+2}\\
\leq t-s + D\cdot (t-s) + 4D\cdot (t-s) 
\leq (5D+1)\cdot d_\Gamma((s,y),(t,y'))
\end{multline*}

In the unchecked case, we always have $2^{n} \leq s \leq 2^{n+1}$ and $s \leq t \leq 2^{n+2}$ for some~$n$. Number $t$ can be described as a convex combination $t=\theta_t^0 \cdot 2^n + \theta_t^1 \cdot 2^{n+1}  + \theta_t^2 \cdot 2^{n+2}$, where at least one of $\theta_t^0$ and $\theta_t^2$ is zero, and similarly for~$s$. Observe that $|\theta_t^i-\theta_s^i|\leq \frac{t-s}{2^n}$. We obtain the bound from above:
\begin{multline*}
\|\Phi(t,y') - \Phi(s,y)\|
\leq (t-s) + \sum_{i=0}^2 \|\theta_t^i\cdot (\phi_{n+i}(y') - \phi_{n+i}(y)) \| + \|(\theta_t^i - \theta_s^i)\cdot \phi_{n+i}(y) \|  \\
\leq (t-s) + 3\left( \rho_+(d_{2^{n+2}}(y,y')) + \frac{t-s}{2^n} \cdot D\cdot 2^{n+2}\right) \\
{\leq} (t-s) + 3\rho_+(4 d_s(y,y')) + 12D\cdot (t-s) \\
\leq 3\rho_+(4d_\Gamma((s,y),(t,y'))) + (12D+1)d_\Gamma((s,y),(t,y')).
\end{multline*}

Note that $\sum_i |\theta^i_t - \theta^i_s| \geq (t-s)/2^{n+1}$. For $v,w\in \H$ of length $1$ and $\theta\in[0,1]$, the triangle inequality gives: $\|w - \theta v \| \geq \|w-v\| - (1-\theta)\|v\|$ and $\|w - \theta v \| \geq \|w\| - \|\theta v\| = 1-\theta$, and we conclude
$$\|w - \theta v \| \geq \max(\|w - v \| - (1-\theta),\ 1-\theta) \geq 1/2 \cdot \|w - v \|.$$
We can now show the bound from below:
\begin{multline*}\|\Phi(t,y') - \Phi(s,y)\|^2 - (t-s)^2
= \sum \|\theta_t^i\cdot \phi_{n+i}(y')  -  \theta_s^i\cdot \phi_{n+i}(y) \|^2 \\
\geq \sum \left( \max(\theta_t^i,\ \theta_s^i) \cdot 1/2 \cdot \|\phi_{n+i}(y') - \phi_{n+i}(y)\|\right)^2 \\
\geq \sum \left( \max(\theta_t^i,\ \theta_s^i) \cdot 1/2 \cdot \rho_-(d_{2^n}(y,y')) \right)^2
\geq \frac{1}{8}\cdot \rho_-\left(\frac{1}{2}d_s(y,y')\right)^2.\qedhere
\end{multline*}
\end{proof}

Note that the only place, where we used the fact that we are dealing with a Hilbert space, was when we used \cite{growthOfCocycles} to deduce that $\rho_+$ can be assumed to be affine. Hence, we have the following.

\begin{cor}\label{C:p-embeddings} Under the additional requirement that $\rho_+$ is affine, Lemma \ref{L:coneEmbeddability} holds for any $\ell_p$, $1\leq p < \infty$.
\end{cor}

\begin{longrem}\label{R:automaticallyAffine} Recall that any coarse map from a quasi-geodesic metric space admits an affine function $\rho_+$. Hence, the additional assumption of Corollary \ref{C:p-embeddings} is automatically satisfied, whenever a cone (its sections) is quasi-geodesic. This holds in particular when $Y$ is a connected manifold.
\end{longrem}

As it is clear from the proof of Lemma \ref{L:coneEmbeddability}, if $\rho_\pm$ functions for the sections are affine, then also our estimates of $\rho_\pm$ functions for the cone are affine.

\begin{cor}\label{C:quasi-isometriesPreserved} Lemma \ref{L:coneEmbeddability} and Corollary \ref{C:p-embeddings} hold also in the case of quasi-isometric embeddings instead of general coarse embeddings.
\end{cor}

%---------------------------------------------------------

\section{Property A}\label{S:A}

By $\ell_1(X)_{1,+}$ we will denote the set of non-negative functions on $X$ summing to $1$. Let us recall the Hulanicki--Reiter condition, which is equivalent to property $A$ for bounded geometry metric spaces \cite{cones}.

\begin{defi}
A metric space $X$
satisfies the Hulanicki--Reiter condition
if for all $R<\infty$, $\eps>0$ there is $S<\infty$ and a map $a\colon X\to \ell_1(X)_{1,+}$ such that $\|a(x)-a(x')\|\leq \eps$ whenever $d(x,x')\leq R$ and the support of $a(x)$ lies in the $S$-ball about $x$.
\end{defi}

The constant $S=S(R,\eps)$ from the above definition will be called the \emph{localisation constant}.

\begin{prop}\label{P:propA} The warped cone $\cO_\Gamma Y$ satisfies the Hulanicki--Reiter condition 
 if and only if all of its sections satisfy the Hulanicki--Reiter condition 
with uniform localisation constants.
\end{prop}
\begin{proof} \emph{``Only if'' part}. Let $R<\infty$, $\eps>0$ and $a\colon \cO_\Gamma Y \to \ell_1(\cO_\Gamma Y)_{1,+}$ be the corresponding Hulanicki--Reiter function 
such that the support of $a(x)$ lies within the $S_{\cO_\Gamma Y}$-ball about $x$.

For $(s,y)\in \cO_\Gamma Y$ let $b(s,y)\in\ell_1(\{s\}\times Y)_{1,+}$ be defined by the formula 
$$b(s,y)(s,y') = \sum_{t\geq 0} a(s,y)(t,y').$$
Clearly $b$ satisfies $\|b(s,y)-b(s,y')\|_1\leq \|a(s,y)-a(s,y')\|_1 \leq \eps$ provided that $d_\Gamma((s,y),(s,y'))\leq R$. Moreover, by concavity of $t\mapsto d_t(y,y')$, we have $d_{t}\geq \frac{t}{s}d_s$ for $t\leq s$ and we conclude that
$$d_\Gamma((s,y),(t,y'))=s-t+d_{t}(y,y') \geq \frac{s-t}{s}\cdot s + \frac{t}{s}\cdot d_s(y,y') \geq {1 \over 2} \cdot d_s(y,y')$$
and thus the support of $b(s,y)$ is contained in the $(2\cdot S_{\cO_\Gamma Y})$-ball about $(s,y)$ in $\{s\}\times Y$ with respect to the warped metric $d_s$ (as the last inequality holds trivially for $t\geq s$).

\emph{``If'' part}. Now, suppose that for $\eps>0$, $R>1$, function $c_n\colon \{n\}\times Y \to \ell_1(\{n\}\times Y)_{1,+}$ is the Hulanicki--Reiter map and $S$ is the corresponding constant. Let $M = \ceil{\eps^{-1}R}$. For $r\geq M$ define $$b(r,y) = \sum_{m=\floor{r}-M+1}^{\floor{r}}M^{-1}c(m,y).$$

To handle small $r$ we will mimic the technique from \cite{cones}*{Lemma 2.5}. Let $\phi\colon[1,\infty)\to[0,1]$ be $M^{-1}$-Lipschitz, equal to $1$ on $[1,M]$ and zero for arguments larger than $2M$. We choose any $y_0\in Y$ and define
$$a(r,y)=\phi(r)\delta_{(1,y_0)} + (1-\phi(r)) b(r,y).$$

First note that $\supp b(r,y)$ lies within $(M+S)$-ball about $(r,y)$ in $\cO_\Gamma Y$. Thus the support of $a(r,y)$ lies within the ball of radius $$\max(M+S,\ 2M+\diam(Y_1,d_1))$$ about $(r,y)$.

It is left to check that $a$ varies by at most $7\eps$ if arguments vary by at most $R$. By the form of the warped metric $d_\Gamma$ it is enough to consider the following two cases separately. Assuming that $d_r(y,y')\leq R$ we have:
\begin{multline*}
\|a(r,y)-a(r,y') \| = (1-\phi(r))\|b(r,y)-b(r,y')\|\\
 \leq \sum_{m=\floor{r}-M+1}^{\floor{r}}M^{-1}\|c(m,y)-c(m,y')\| \leq \sum_{m=\floor{r}-M+1}^{\floor{r}}M^{-1}\eps = \eps
\end{multline*}
and if $k\leq R$ one obtains:
\begin{multline*}
\|a(r+k,y)-a(r,y) \| \leq \|\phi(r+k)\delta_{(1,y_0)}-\phi(r)\delta_{(1,y_0)}\| \\
+ \|(1-\phi(r+k))b(r+k,y) - (1-\phi(r))b(r+k,y)\| \\
+ \|(1-\phi(r))b(r+k,y)-(1-\phi(r))b(r,y)\| \leq \frac{k}{M} + \frac{k}{M} + \frac{2k+2}{M} \leq 6\eps.\qedhere
\end{multline*}
\end{proof}

%---------------------------------------------------------

\section{Coarse properties of the cone imply equivariant properties of the group}\label{S:metricVsEquivariant}

In this section we will show how to strengthen \cite{cones}*{Proposition 4.1, 4.4} to the following. 

\begin{prop}\label{P:AtoAmenability} Let $\mu$ be a $\Gamma$-invariant measure on $Y$ and assume that there exists a subset $P\subseteq Y$ of positive measure on which the action of $\Gamma$ is free.
\begin{enumerate}
\item\label{AtoAmenability:cond1} If $\cO_\Gamma Y$ has property A, then $\Gamma$ is amenable.
\item\label{AtoAmenability:cond2} If $\cO_\Gamma Y$ embeds coarsely into a Hilbert space, then $\Gamma$ has the Haagerup property.
\end{enumerate}
\end{prop}
In the original work, $Y$ was required to be a compact Lie group containing $\Gamma$ as a dense subgroup.

We will need the following consequence of property A, which is equivalent to property A for bounded geometry metric spaces, see \cites{cones,Roe}. Recall that a controlled set in the product $X\times X$ is a set lying in a bounded distance from the diagonal.

\begin{fact}\label{F:positiveKernels} If a metric space $X$ has property A, then there is a sequence $(k_n)$ of continuous positive-type kernels such that:
\begin{enumerate}
\item $|k_n(x,x')|\leq 1$;
\item each $k_n$ has controlled support;
\item $k_n(x,x')\to 1$ as $n\to\infty$, uniformly on each controlled set.
\end{enumerate}
\end{fact}

The proof of Proposition \ref{P:AtoAmenability} does not construct or require the kernels to be $\Gamma$-invariant and relies only on an invariant measure, which enables extending the result beyond the case of subgroup actions.

\begin{proof}[Proof of Proposition \ref{P:AtoAmenability}]
Without loss of generality $\mu(P)=1$.

Let us start with assertion \eqref{AtoAmenability:cond1}. Let $(k_n)$ be a sequence of positive-type kernels on $\cO_\Gamma Y$ satisfying the conditions of Fact \ref{F:positiveKernels}. We define:
\begin{equation}\label{AtoAmenability:formula}
h_n^t(\gamma) = \int_P k_n((t,y),(t,\gamma y))\ d\mu(y).
\end{equation}
Note that $h_n^t$ is a positive-type kernel on $\Gamma$:

\begin{multline*}
\sum_{\gamma,\gamma'} \lambda_{\gamma}\lambda_{\gamma'} h_n^t(\gamma^{-1}\gamma') 
= \sum_{\gamma,\gamma'} \lambda_{\gamma}\lambda_{\gamma'} \int_P k_n\Big((t,y),(t,\gamma^{-1}\gamma' y)\Big)\ d\mu(y) \\
= \int_P \sum_{\gamma,\gamma'} \lambda_{\gamma}\lambda_{\gamma'} k_n\Big((t,(\gamma')^{-1} z),(t,\gamma^{-1} z)\Big)\ d\mu(z) \geq 0
\end{multline*}
-- in the second equality we change the variables with $z=\gamma' y$ using the $\Gamma$-invariance of $\mu$ and $P$ and, further, change the order of summation and integration.

Since $h_n^t$ are bounded and with a countable domain, one can choose a sequence $(t_j)_{j=1}^{\infty}$ converging to infinity such that $h_n^{t_j}$ tends to some $h_n$ as $j$ goes to infinity. Clearly $h_n$ is a positive-type function on $\Gamma$.

Kernel $k_n$ has controlled support, thus there is some $N<\infty$ such that $d_t(y,y')>N$ implies $k_n((t,y), (t,y'))=0$.
Since the action is free, we have $\lim_{t\to \infty} d_t(y, \gamma y) = |\gamma|$ (Remark \ref{r:metricConvergence}). Hence, whenever $|\gamma|>N$, we obtain by the Lebesgue dominated convergence theorem:
$$h_n(\gamma) = \lim_j \int_P k_n((t_j,y),(t_j,\gamma y))\ d\mu(y) = 0,$$
so $h_n$ has a compact support.

Similarly, one can check that $h_n(\gamma)\to 1$ for every $\gamma$. Indeed, functions $k_n$ converge to $1$ uniformly on controlled sets, thus for each $\eps>0$, there is $N=N(\eps,|\gamma|)$ such that $k_n((t,y),(t,y'))\geq 1-\eps$ if $d_t(y,y')\leq |\gamma|$ and $n\geq N$. Clearly, $d_t(y,\gamma y)\leq |\gamma|$, so $1\geq h_n^t(\gamma)\geq 1-\eps$ whenever $n\geq N$, and, hence, the same is true for $h_n$.

Summarising, $\Gamma$ admits a sequence of positive-type functions with compact support tending pointwise to $1$ and this is one of the characterisations of amenability.

For assertion \eqref{AtoAmenability:cond2}, let $\map k {(\cO_\Gamma Y)^2} [0,\infty)$ be a continuous negative-type kernel establishing the coarse embeddability. That is, given a continuous coarse embedding $\map f {\cO_\Gamma Y} \H$, one can define such a kernel by the formula $k(x,x') = \|f(x)-f(x')\|^2$ (see \cite{Roe}*{Section 11.2}). For some nondecreasing and unbounded maps $\rho_-$, $\rho_+$, we have:
$$\rho_-\circ d_\Gamma(x,x') \leq k(x,x') \leq \rho_+\circ d_\Gamma(x,x').$$
By formula \eqref{AtoAmenability:formula} with $k$ instead of $k_n$, one can define a negative-type function $h^t$ satisfying $0\leq h^t(\gamma) \leq \rho_+(|\gamma|)$. Again, by choosing an appropriate sequence $(t_j)$ we obtain kernel $h=\lim_j h^{t_j}$, which is additionally proper as $\rho_-(|\gamma|)\leq h(\gamma)$, which certifies the desired Haagerup property.
\end{proof}

A reader familiar with the notion of asymptotic embedding \cite{randomGraphs} (asymptotically conditionally negative definite kernel of \cite{BGW16}), can immediately deduce from the proof that the assumptions of Proposition \ref{P:AtoAmenability} \eqref{AtoAmenability:cond2} can be relaxed.

\begin{cor} Under the assumptions of Proposition \ref{P:AtoAmenability}, if  $\cO_\Gamma Y$ admits an asymptotic embedding into a Hilbert space, then $\Gamma$ has the Haagerup property.
\end{cor}

Furthermore, the existence of a fibred coarse embedding \cite{fibred} implies the existence of an asymptotic embedding  \cite{BGW16}, so, in particular, we obtain the following.

\begin{cor} Under the assumptions of Proposition \ref{P:AtoAmenability}, if  $\cO_\Gamma Y$ admits a fibred coarse embedding into a Hilbert space, then $\Gamma$ has the Haagerup property.
\end{cor}

%---------------------------------------------------------

\section{Embeddable cones without property A}\label{S:main}

Let us recall some facts about profinite completions of groups.

Throughout this section $\Gamma$ will always be a finitely generated group and $(\Gamma_n)_{n\geq 1}$ will be a decreasing sequence of its finite index normal subgroups.

Denote by $G_n$ the quotient group $\Gamma/\Gamma_n$, and let $G_0=\Gamma/\Gamma$ be the trivial group. We have the following sequence of epimorphisms $1 \twoheadleftarrow G_1 \twoheadleftarrow G_2 \twoheadleftarrow \ldots$. The homomorphism $G_n\onto G_{n-1}$ will be denoted by $f_n$. We will call the inverse limit of this sequence the completion of $\Gamma$ with respect to $(\Gamma_n)$ (it is also sometimes called ``the boundary of the coset tree'') and denote it by 
$\widehat\Gamma((\Gamma_n))$. It refines the notion of the profinite completion $\widehat\Gamma$ of a group $\Gamma$, which is obtained if one defines $\Gamma_n$ as the intersection of all subgroups of $\Gamma$ of index at most $n$.

Completion $\widehat\Gamma((\Gamma_n))$ can be viewed as the following subset of the Cartesian product $\prod_n G_n$:
$$G = \{(g_n)_{n\geq 0} \st g_{n-1}=f_n(g_n)\ \forall n\geq 1 \}.$$
Note that this is a compact subgroup in the product. Observe also that the product of quotient maps $\Gamma\to G_n$ gives a homomorphism from $\Gamma$ to $G$ with dense image. The kernel of this map is $\bigcap_n \Gamma_n$. In particular, if the intersection is trivial, then the action of $\Gamma$ on $G$ by left multiplication is free.

Define a metric on $\widehat\Gamma((\Gamma_n))$ by:
$$d((g_n), (g'_n)) = \max \{ a_n \cdot d_\mathrm{bin}(g_n, g_n') \},$$
where $a_n$ is a tending to zero sequence of positive reals and $d_\mathrm{bin}$ is the discrete $\{0,1\}$-valued metric.

\begin{rem} Completion $\widehat\Gamma((\Gamma_n))$ admits a $\Gamma$-invariant measure (the Haar measure) and metric (defined above).
\end{rem}

Let us recall the original definition of property A introduced in \cite{Yu}.

\begin{defi} A metric space $X$ has property A, if for any $\eps>0$ and $R<\infty$, one can associate to every point $x\in X$ a finite set $A_x \subseteq X\times \N$ in such a way that:
$$\frac{|A_x\triangle A_y|}{|A_x\cap A_y|} < \eps$$
for any $x,y\in X$ with $d(x,y)<R$ (where $\triangle$ denotes the symmetric difference) and there is a constant $S=S(\eps,R)$ such that $A_x\subseteq B(x,S)\times \N$ for every $x\in X$.
\end{defi}

In what follows, on $G_n$ we will consider the right invariant word metric $d_{G_n}$ related to the image of $S$ under the quotient map; in other words, the distance between $g_n$ and $g_n'$ is the length of the shortest $\gamma\in \Gamma$ such that $\gamma g_n = g_n'$.

The next theorem is a generalisation of the well-known result for box spaces. The author would like to acknowledge that it was Piotr Nowak who suggested that all the conditions are equivalent.

\noindent\begin{minipage}{\textwidth}
\begin{thm}\label{T:A=Amen} Assume that the sequence $(\Gamma_n)_{n\geq 1}$ intersects trivially. The following conditions are equivalent:
\begin{enumerate}[label={\roman*)}, ref={\roman*},itemsep=.76ex]
\item\label{A=Amen:cond1} $\Gamma$ is amenable;
\item\label{A=Amen:cond2'} $\cO_\Gamma \widehat\Gamma$ has property A;
\item\label{A=Amen:cond2} $\cO_\Gamma \widehat\Gamma((\Gamma_n))$ has property A;
\item\label{A=Amen:cond3} $\bigsqcup_n G_n$ has property A.
\end{enumerate}
\end{thm}
\end{minipage}
\begin{proof} The equivalence \eqref{A=Amen:cond1}$\iff$ \eqref{A=Amen:cond3} is a well known fact, see \cite{Roe}.

Implications \eqref{A=Amen:cond2'}$\implies$\eqref{A=Amen:cond1} and \eqref{A=Amen:cond2}$\implies$\eqref{A=Amen:cond1} follow immediately from Proposition~\ref{P:AtoAmenability}.

For the converse implications, the assumption $\bigcap_n \Gamma_n = \{1\}$ is not needed. We will use a version of \cite{cones}*{Proposition 3.1}, which asserts that if
$(Z,d^Z)$ has property A and admits a coarse action of an amenable group $\Delta$, then the warped space $(Z,d^Z_\Delta)$ also has property A.

The assumptions are indeed satisfied for $\cO\widehat\Gamma((\Gamma_n))$ (in particular, for  $\cO\widehat\Gamma$). By an argument similar to that in Proposition \ref{P:propA}, it suffices to check that each space $(\widehat\Gamma,sd)$ has property $A$ with a uniform bound on localisation constants and cardinalities of the sets from the definition of property A. Fix $R$. Let $n$ be the greatest integer such that $sa_n \geq R$. For $g=(g_i)$, we define $A_g$ simply as the singleton of $(g_1,\ldots, g_n, h_{n+1}, h_{n+2}, \ldots)$, where $h_m$ are chosen in the same way for all $g$ with the first $n$ coordinates equal. Clearly, if $sd(g,g') < R$, then $A_{g}=A_{g'}$, and we also have $A_g \subseteq B(g,sa_{n+1})\subseteq B(g,R)$.
\end{proof}

\begin{conv}\label{conv} For the rest of this section, we assume that coefficients in the definition of metric $d$ on $\widehat\Gamma((\Gamma_n))$ satisfy $a_{n+1} < { a_n \over \diam G_{n} }$.
\end{conv}

\begin{thm}\label{T:main}
The following conditions are equivalent:
\begin{enumerate}[label={\roman*}),ref={\roman*}, itemsep=1.5ex]
\item\label{main:cond1} $\bigsqcup_n G_n$ embeds coarsely (resp. quasi-isometrically) into a Hilbert space;

\item\label{main:cond2} $\cO_\Gamma\widehat\Gamma((\Gamma_n))$ embeds coarsely (resp. quasi-isometrically) into a Hilbert space.
\end{enumerate}
Furthermore, if any of the above conditions holds and sequence $(\Gamma_n)$ intersects trivially, then $\Gamma$ has the Haagerup property.
\end{thm}

\begin{proof}
\emph{``Furthermore'' part.}
Since the $\Gamma$-action on $\widehat\Gamma((\Gamma_n))$ is free, condition \eqref{main:cond2} implies the Haagerup property by Proposition \ref{P:AtoAmenability}. Alternatively, the fact that condition \eqref{main:cond1} implies the Haagerup property of $\Gamma$ if the intersection of $(\Gamma_n)$ is trivial is well-known, see \cite{Roe}.

\emph{Implication \eqref{main:cond1}$\implies$\eqref{main:cond2}.} By the assumption, there are nondecreasing functions $\rho_-, \rho_+\colon [0,\infty)\to \R$ tending to infinity and for each $n$ there is a map $\phi_n\colon G_n \to \H$ satisfying for all $\gamma,\gamma'\in G_n$:
$$\rho_-\circ d_{G_n}(\gamma, \gamma') \leq \|\phi_n(\gamma)-\phi_n(\gamma') \| \leq \rho_+\circ d_{G_n}(\gamma, \gamma'),$$
where $d_{G_n}$ is the \emph{right} invariant metric on $G_n$. In case of quasi-isometric embeddings, we require $\rho$ functions to be affine and increasing.

By Lemma \ref{L:coneEmbeddability} and Corollary \ref{C:quasi-isometriesPreserved} it is enough to construct embeddings of $(\widehat\Gamma((\Gamma_n)),d_s)$ with uniform estimates. Observe that $d(g,g') = a_{n(g,g')}$, where $n(g,g')$ is the minimal index such that $g_n\neq g'_n$ (we put $a_\infty=0)$.

Fix $s \geq 1$.
Let us denote $\delta_n(g,g')=d_{G_n}(g_n,g'_n)$.
Since the action is isometric, by Proposition \ref{P:quasimetric} we obtain:
\begin{multline}\label{main:equality}
d_s(g,g') = \min_\gamma (|\gamma| + sd(\gamma g, g')) \\
= \min_\gamma (|\gamma| + sa_{n(\gamma g, g')})=\inf_n(\delta_{n-1}(g,g') + sa_n).
\end{multline}
(the passage from minimum to infimum is necessary if $g'=\gamma g$ for some $\gamma$; then we also have $n(\gamma g', g)=\infty$).

Since $n\mapsto \delta_n(g,g')$ is increasing and $n\mapsto sa_n$ is decreasing, we get the following (for any $k\in \N$):
\begin{multline}\label{main:inequality}
d_s(g,g')
= \min\Big(\inf_{n\leq k}(\delta_{n-1}(g,g') + sa_n),\ \inf_{n>k}(\delta_{n-1}(g,g') + sa_n) \Big) \\
\geq \min(sa_{k},\ \delta_{{k}}(g,g'))
\end{multline}
Let $N$ be the largest integer such that 
$sa_N\geq 1$.
By Convention \ref{conv} we obtain $sa_{N-1} \geq \diam G_{N-1}\geq \delta_{N-1}(g,g')$, so formula \eqref{main:inequality} for $k=N-1$ boils down to:
\begin{equation}\label{main:secondInequality}
d_s(g,g') \geq \delta_{N-1}(g,g')
\end{equation}
Adding inequality \eqref{main:inequality} for $k=N$ and the above inequality \eqref{main:secondInequality} we obtain:
\begin{equation}\label{main:thirdInequality}
2d_s(g,g')\geq \delta_{N-1}\left(g,g'\right) + \min\left(sa_N,\ \delta_N(g,g')\right) \eqdef d_s'(g,g')
\end{equation}
and by formula \eqref{main:equality} we get the converse estimate:
\begin{align*}
d_s(g,g')&\leq
\min\left(\delta_{N-1}\left(g,g'\right) + sa_N,\ \delta_{N}(g,g') + sa_{N+1}\right) \\
&\leq
\delta_{N-1}\left(g,g'\right) + \min\big(sa_N,\ \delta_{N}(g,g')\big) + 1 \\
&= d_s'(g,g')+1,
\end{align*}
which means that the pseudometric $d_s'$ is quasi-isometric to metric $d_s$ with constants not depending on $s$.

Consequently (compare Appendix 
\ref{appendix}), it is enough to coarsely embed $\widehat\Gamma((\Gamma_n))$ separately with respect to the pseudometrics $d^1_s(g,g')=\delta_{{N-1}}(g,g')$ and $d^2_s(g,g')=\min\big(\delta_{N}(g,g'),\ sa_N\big)$. In the first case we can use the embedding $\map {\phi_{N-1}} {G_n} \H$ from the assumptions directly, by composing it with the quotient map $\widehat\Gamma((\Gamma_n))\to G_n$.

The case of $d_s^2$ is more involved. It is a matter of simple case analysis (under the assumption that $\rho_-(r)\leq r \leq \rho_+(r)$) to verify that:
\begin{multline*}
\rho_-\circ \min\big(\delta_{N}(g,g'),\ sa_N\big)
\leq \min\big(\|\phi_N(g)-\phi_N(g') \|,\ sa_N\big) \\
\leq \rho_+\circ \min\big(\delta_{N}(g,g'),\ sa_N\big),
\end{multline*}
which means that if we consider the metric $d_\H'$ on the Hilbert space $\H$ equal to the minimum of the norm metric and constant $sa_N$, then $\phi_N$ is a coarse embedding with respect to $d_s^2$ and control functions $\rho_-$ and $\rho_+$. 

Hence, as we know that pseudometric $d^2_s$ is coarsely equivalent to the pseudometric $d_\H'$ (considered as a metric on $\widehat\Gamma((\Gamma_n))$), it suffices to check that $\widehat\Gamma((\Gamma_n))$ can be embedded into a Hilbert space (with the standard metric) in a bi-Lipschitz way with respect to $d_\H'$.

For $r, l>0$ let $m_{l}(r) = l(1-\exp(-r/l))$. We have:
$$
{\min(r,l) \over m_{l}(r)} =
{l \min(r/l,1) \over l(1-\exp(-r/l))} =
{\min(r/l,1) \over 1-\exp(-r/l)},
$$
and thus these quotients are bounded:
$$1 \leq {\min(x,1) \over 1-e^{-x}} \leq \sup_{x\in(0,\infty)} {\min(x,1) \over 1-e^{-x}} = \sup_{x\in(0,1]} {\min(x,1) \over 1-e^{-x}} = {1 \over 1-e^{-1}}.$$

Denote by $k_N$ the negative-type kernel associated with the embedding $\phi_N$ given by $k_N(g,g') = \|\phi_N(g)-\phi_N(g')\|^2$.
Since $m_l$ is a Bernstein function, kernel $m_l\circ k_N$ is still negative-type by the Schoenberg's lemma (e.g. \cite{Roe}*{Proposition 11.12}) and, as such, yields an embedding $\varphi_N$ into a Hilbert space. For $l=(sa_N)^2$ we have:
$${\|\varphi_N(g) - \varphi_N(g')\|^2 \over d_H'(\phi_N(g), \phi_N(g'))^2}
= {{m_l\circ k_N(g,g')} \over {\min(k_N(g,g'),\ l)}}
\in \left[{1-e^{-1}},\ 1 \right],$$
which confirms the aforementioned bi-Lipshitz equivalence.

\emph{Implication \eqref{main:cond2}$\implies$\eqref{main:cond1}.} Fix $n$ and let $s(n)=\frac{\diam G_n}{a_n}$. Then, $s(n)a_{n} \geq \diam G_n$ and $s(n)a_{n+1}<1$, so from formulas \eqref{main:secondInequality} and \eqref{main:equality} we get
$$\delta_n \leq d_{s(n)} \leq \delta_n + sa_{n+1} \leq \delta_n + 1,$$
so $s(n)$-th section of the warped cone is $(1,1)$-quasi-isometric to $G_n$, which ends the proof.
\end{proof}

Coarse embeddability in $\ell^p$ for $p\in[1,2]$ is equivalent, so if we do not insist on quasi-isometric embeddings, then we can put any such space (or even two different) in Theorem \ref{T:main}. However, it does not resolve the following problem.

\begin{question} Does the equivalence from Theorem \ref{T:main} hold for $\ell_p$ spaces with $p>2$?
\end{question}

Note that $\frac{s(n+1)}{s(n)} = \frac{\diam G_{n+1}}{a_{n+1}} \cdot \frac{a_n}{\diam G_n} > \diam G_{n+1}$, meaning that $\{s(n) \st n\in \N\}$ is a coarse disjoint sum of points. Thus, from the proof of implication \eqref{main:cond2}$\implies$\eqref{main:cond1} above, one immediately concludes the following corollary.

\begin{cor}\label{C:embeddedBoxes}
The subspace $\bigcup_n \{s(n)\} \times \widehat\Gamma((\Gamma_n)) \subseteq \cO_\Gamma \widehat\Gamma((\Gamma_n))$ is (1,1)-quasi-isometric to a coarse disjoint sum $\bigsqcup_n G_n$.
\end{cor}

An immediate consequence of Theorems \ref{T:A=Amen} and \ref{T:main} is the following.

\begin{thm}\label{T:summary} Let $\Gamma_n \triangleleft\Gamma$ be a residual chain of finite index subgroups in a non-amenable group $\Gamma$ and assume that the coarse disjoint sum $\bigsqcup_n G_n$ embeds coarsely into a Hilbert space, where $G_n = \Gamma\slash \Gamma_n$.

Then, the warped cone $\cO_{\Gamma} \widehat\Gamma((\Gamma_n))$ embeds coarsely into a Hilbert space, and yet does not have property A.
\end{thm}

Examples of such sequences $(\Gamma_n)$ for free groups are the main results of \cites{AGS, Khukhro}.
These examples can be parametrised by an integer $m\geq 2$ (and also by the number $d\geq 2$ of free generators of $\Gamma=\F_d$). The case of $m=2$ is due to \cite{AGS} and the generalisation to $m>2$ was obtained in \cite{Khukhro}.
 Furthermore, we have permanence results for such sequences. By \cite{Khukhro12}, the semidirect product $\Gamma \rtimes \Delta$ of $\Gamma$ as above and an amenable, residually finite $\Delta$ also admits a sequence of quotients with the desired properties. By \cite{CDK}, the same holds for any wreath product $A\wr \Gamma$ with abelian $A$ (note that the published version \cite{CD} of this paper does not contain the relevant result).

%---------------------------------------------------------

\appendix
\section{Embeddings of product spaces}\label{appendix}

In the proof of Theorem \ref{T:main}, we used the fact that coarse embeddings into a Hilbert space of both Cartesian factors yield a coarse embedding of the Cartesian product (actually, we used the fact that if a metric is a sum of pseudometrics, then embeddings with respect to the pseudometrics give an embedding with respect to the metric). Below, we prove a uniform version of this fact for arbitrary $n$-fold Cartesian products in order to answer a problem of Nowak from \cite{Nowak}.

First, we need a lemma.

\begin{lem}\label{L:concave} Let $f,F\colon\R_+\to\R_+$ be nondecreasing and unbounded. If $f$ satisfies $\liminf_{s\to 0}{f(s)\over s} > 0$, then there is an unbounded and nondecreasing function $c\leq f$ such that $c(\theta r)\geq \theta c(r)$ for all $\theta\in[0,1]$ and $r\in [0,\infty)$. Similarly, if $F$ satisfies $\limsup_{s\to 0}{F(s)\over s}<\infty$, we have $C\geq F$ with $C(\theta r)\leq \theta C(r)$.
\end{lem}

\begin{proof} Let 
\begin{eqnarray*}
m_t &=& \inf_{0 < s \leq t}{f(s)\over s} > 0, \\
M_t &=& \sup_{0 < s \leq t}{F(s)\over s} < \infty
\end{eqnarray*}
and define $c(t)=m_t\cdot t$ and $C(t)=M_t\cdot t$. We have $c(t)=m_t\cdot t \leq {f(t)\over t} \cdot t = f(t)$, and similarly $C(t)\geq F(t)$. Moreover, $c(\theta t) = m_{\theta t}\cdot \theta t \geq m_t \cdot \theta t = \theta c(t)$ and similarly for $C$. The coefficient $M_t$ is nondecreasing, so $C$ is increasing, and, as $C$ is greater than $F$, it tends to infinity. It is left to show the same for $c$.

Consider $t<t'$. Since $f$ is nondecreasing, we have:
$$m_{t'} = \min\left(m_t, \inf_{t \leq s \leq t'} {f(s)\over s}\right) \geq \min\left(m_t, {f(t)\over t'}\right)$$
and by multiplying both sides by $t'$ we obtain:
\begin{equation}\label{L:concave:inequality}
c(t') \geq \min(m_t\cdot t',\ f(t))\geq c(t),
\end{equation}
that is, $c$ is nondecreasing.

Let now $R<\infty$. Let $t$ be so large that $f(t)\geq R$ and let $t'
\geq \max(t,\ Rm_t^{-1})$. We conclude that function $c$ is unbounded by using inequality \eqref{L:concave:inequality} again: $c(t') \geq \min(m_t\cdot t',\ f(t))\geq R$.
\end{proof}

In some cases the assumptions of the above lemma are trivially satisfied.

\begin{longrem}\label{R:discretenessForRho} In order to guarantee $\liminf_{s\to 0}{f(s)\over s} > 0$, it is enough to assume $f\geq b$ for some positive constant $b$. Similarly, $F\colon [a,\infty)\to \R_+$ can be extended to the interval $[0,a]$ by putting $F(\theta a) = \theta F(a)$, which gives finiteness of $\limsup_{s\to 0}{F(s)\over s}$.
\end{longrem}

We are now ready to prove the main proposition of the appendix.

\begin{prop}\label{P:productEmbedding} Let $a,b>0$, $n\in \N$, $(X_i,d_i)_{1\leq i\leq n}$ be metric spaces and $p\in[1,\infty)$. We require $d_i(x,x')\geq a$ for $x\neq x'\in X_i$. Assume that maps $\psi^i_p\colon X_i\to \ell_p$ satisfy
$$\rho_-\circ d_i(x,x')\leq \|\psi^i_p(x) - \psi^i_p(x')\|_p \leq \rho_+\circ d_i(x,x')$$
for some unbounded and nondecreasing $\rho_-,\rho_+\colon[a,\infty)\to [b,\infty)$. Then, there are unbounded and nondecreasing $c=c(\rho_-)$ and $C=C(\rho_+)\colon [a,\infty)\to [b,\infty)$, and a map $\psi_p\colon \prod X_i\to \ell_p$ satisfying:
$$\sqrt[p]{b^{p-1}\cdot c\circ d(x,x')} \leq \|\psi_p(x)-\psi_p(x')\|_p \leq C\circ d(x,x'),$$
where $x=(x_i)$, $x'=(x'_i)\in \prod X_i$ and $d(x,x')=\sum d_i(x_i, x'_i)$.
\end{prop}
\begin{proof}
We take $\psi_p = \bigoplus_i\psi^i_p\colon \prod_i X_i \to \bigoplus_i \ell_p$ and define $c$, $C$ as functions from Lemma \ref{L:concave} for $f=\rho_-$ and $F=\rho_+$. Recall that $c(\theta r)\geq \theta c(r)$ and $C(\theta r)\leq \theta C(r)$ for all $\theta\in[0,1]$.
We can estimate:
\begin{multline*}
\|\psi_p(x)-\psi_p(x')\|^p_p = \sum_i \|\psi^i_p(x_i)-\psi^i_p(x_i')\|^p_p
\geq \sum_i \big(c\circ d_i(x_i,x_i')\big)^p \\
= b^p \sum_i \left({c\circ d_i(x_i,x_i')\over b}\right)^p
\geq b^p \sum_i {c\circ d_i(x_i,x_i')\over b}
= b^{p-1} \sum_i c\circ d_i(x_i,x_i') \\
\geq  b^{p-1} \sum_i {d_i(x_i,x_i')\over d(x,x')} \cdot c\circ d(x,x')
= b^{p-1} \cdot c\circ d(x,x').
\end{multline*}

The estimates from above are straightforward:
\begin{multline*}
\|\psi_p(x)-\psi_p(x')\|_p \leq \sum_i \|\psi^i_p(x_i)-\psi^i_p(x_i')\|_p
\leq \sum_i C\circ d_i(x_i,x_i') \\
\leq  \sum_i {d_i(x_i,x_i')\over d(x,x')} \cdot C\circ d(x,x') = C\circ d(x,x').\qedhere
\end{multline*}
\end{proof}

\begin{longrem} For $\ell_1$, the fact that $c(r)\geq b$ for $r>0$ is not actually utilised in the proof. Consequently, we also do not need to require $d_i(x,x')\geq a$ and it is enough to assume $\liminf_{t\to 0}\frac{\rho_-(t)}{t} > 0$ and $\limsup_{t\to 0}\frac{\rho_+(t)}{t} <\infty$.
\end{longrem}

In \cite{Nowak}, Nowak proves that a coarse disjoint sum of increasing Cartesian powers of the two-element group does not have property A, and hence -- as property A is preserved by subspaces -- any disjoint sum of Cartesian powers of any nontrivial space $X$ does not have property A. He states the result for $X$ being a nontrivial finitely generated amenable group, as in this case the result can be shown directly, without the subspace argument. He observes that if $X$ is finite or -- more generally -- embeds bi-Lipschitz in $\ell_1$, then the disjoint sum of its powers embeds coarsely into $l_p$ for $1\leq p \leq \infty$. Thereby, he established examples of spaces without property A, yet embeddable into all $\ell_p$.

He asks what happens for amenable groups which do not embed quasi-isometrically into a Hilbert space \cite{Nowak}*{Remark 5.3}.

Our proposition guarantees that if a uniformly discrete space $X$ embeds \emph{coarsely} into $\ell_1$ (in particular, if $X$ is any finitely generated amenable group or, even more generally, has property A), then its Cartesian powers do so, and, thus, the disjoint sum of the powers, too. As $\ell_1$ embeds coarsely into any $\ell_p$, this shows that Nowak's theorem holds not only for amenable groups with a quasi-isometric embedding, but for all uniformly discrete spaces embeddable in $\ell_1$.

%---------------------------------------------------------

\section*{Acknowledgment}
The author is grateful to Piotr Nowak for introduction to the subject, encouragement and remarks on the manuscript. The author also wishes to express his gratitude to Martin Finn-Sell, Ana Khukhro, Piotr Nowak, and Damian Osajda for inspiring conversations.

The author would like to thank Goulnara Arzhantseva and Łukasz Garncarek for valuable comments on the manuscript.

The author was partially supported by Fundacja na rzecz Nauki Polskiej grant MISTRZ 5/2012 of Prof. Tadeusz Januszkiewicz and by Narodowe Centrum Nauki grant DEC-2013/10/EST1/00352 of Dr Piotr Nowak.

\begin{center}
Preprint of an article accepted by Journal of Topology and Analysis\\DOI 10.1142/S179352531850019X © World Scientific Publishing Company \\http://www.worldscientific.com/worldscinet/jta
\end{center}

%---------------------------------------------------------

\end{document}